\documentclass[10pt,reqno]{amsart}
\usepackage{amsmath,amssymb,latexsym,soul,cite,mathrsfs}

\usepackage{color,enumitem,graphicx}

\usepackage[hyperpageref]{backref}

\usepackage[hyperpageref]{backref}

\usepackage[colorlinks=true,urlcolor=blue,
citecolor=red,linkcolor=blue,linktocpage,pdfpagelabels,
bookmarksnumbered,bookmarksopen]{hyperref}
\usepackage[english]{babel}
\usepackage{enumitem}

\usepackage[left=2.0cm,right=2.0cm,top=2.5cm,bottom=2.5cm]{geometry}

\makeatletter
\newcommand{\leqnomode}{\tagsleft@true}
\newcommand{\reqnomode}{\tagsleft@false}
\makeatother

\numberwithin{equation}{section}

\newtheorem{theorem}{Theorem}[section]
\newtheorem{lemma}[theorem]{Lemma}
\newtheorem{corollary}[theorem]{Corollary}
\newtheorem{proposition}[theorem]{Proposition}

\newtheorem{remark}[theorem]{Remark}
\newtheorem{definition}[theorem]{Definition}

\providecommand{\ln}{\mathop{\rm ln}\nolimits}

\numberwithin{equation}{section}
\newcommand{\R}{\mathbb R}

\title[Regularity results for quasilinear elliptic problems]{Regularity results for quasilinear elliptic problems driven by the fractional $\Phi$-Laplacian operator}

\author[M.L. Carvalho]{M. L. Carvalho}
\author[E. D. Silva]{E. D. Silva}
\author[J.C. \ de Albuquerque]{J. C. de Albuquerque}
\author[S. Bahrouni]{Sabri Bahrouni}

\address[M.L. Carvalho]{Department of Mathematics, Federal University of Goi\'{a}s
	Federal University of Goi\'{a}s
	\newline\indent
	74001-970, Goi\'{a}s-GO, Brazil}
\email{\href{marcos_leandro_carvalho@ufg.br}{marcos$\_$leandro$\_$carvalho@ufg.br}}

\address[E.D. Silva]{Department of Mathematics, Federal University of Goi\'{a}s}
\email{\href{mailto: edcarlos@ufg.br}{edcarlos@ufg.br}}

\address[J.C. de~Albuquerque]{Department of Mathematics, Federal University of Goi\'{a}s}
\email{\href{mailto: josecarlos.melojunior@ufpe.br}{josecarlos.melojunior@ufpe.br}}

\address[S. Bahrouni]{Mathematics Department, Faculty of Sciences, University of Monastir, 5019 Monastir, Tunisia}
\email{\href{mailto: sabribahrouni@gmail.com}{sabribahrouni@gmail.com}}

\thanks{Corresponding author: E. D. Silva}

\thanks{Research supported in part by INCTmat/MCT/Brazil, CNPq and CAPES/Brazil. The second author was partially supported by CNPq and with grants 309026/2020-2 and 429955/2018-9.}


\begin{document}
	
	
	\begin{abstract}
		It is established $L^{p}$ estimates for the fractional $\Phi$-Laplacian operator defined in bounded domains where the nonlinearity is subcritical or critical in a suitable sense. Furthermore, using some fine estimates together with the Moser's iteration, we prove that any weak solution for fractional $\Phi$-Laplacian operator defined in bounded domains belongs to $L^\infty(\Omega)$ under appropriate hypotheses on the $N$-function $\Phi$. Using the Orlicz space and taking into account the fractional setting for our problem the main results are stated for a huge class of nonlinear operators and nonlinearities.  
	\end{abstract}
	
	\maketitle
	
	\section{Introduction and main results}
	
	\subsection{Fractional Orlicz Sobolev spaces: an overview}
	
	Quasilinear elliptic problems using Orlicz and Orlicz-Sobolev setting appeared in \cite{orlicz1}. Around 1991 Kov\'{a}\^{c}ik and R\'{a}kosn\'{\i}k published a standard reference for basic properties on Orlicz spaces \cite{KR,gossez1,gossez2}. It is important to emphasize that some generalizations of the classical Sobolev spaces could be considered. Firstly, by using Orlicz-Sobolev spaces, a systematic study of these spaces was initiated in \cite{Adams1, Donaldson, DT} in connection with the analysis of nonlinear partial differential equations without a polynomial growth. Now we turn our attention to the nonlocal problems, which received great attention in the last decades, a fundamental tool to treat these type of problems is the so-called fractional order Sobolev spaces.  Several definitions of fractional Sobolev have been proposed in the literature in particular the fractional Orlicz-Sobolev spaces $W^{s,\Phi}(\Omega)$, see for instance \cite{ABS, BS20, Sabri1, Sabri2, FBS, DNFBS}.

	In the present work, given $N$-function $\Phi$ (see Definition \ref{Musielak function}) we consider the fractional Orlicz Sobolev spaces defined as
	\begin{eqnarray}\label{W}
		W^{s,\Phi}(\mathbb{R}^N)=\left\{u\in L^{\Phi}(\mathbb{R}^N): \mathcal{J}_{s,\Phi}(u)<\infty\right\},
	\end{eqnarray}
	where the usual Orlicz space $L^{\Phi}$ is defined as
	$$
	L^{\Phi}(\mathbb{R}^N) = \left\{u: \mathbb{R}^N\to \mathbb{R}\ \text{measurable, such that}\ \mathcal{J}_{\Phi}(u)<\infty\right\}
	$$
	where 
	\begin{eqnarray}\label{bar-phi}
		\Phi(t):=\int_0^t\varphi(\tau)\tau d\tau,~t\geq 0,
	\end{eqnarray}
	where the modulars $\mathcal{J}_{\Phi}$ and $\mathcal{J}_{s,\Phi}$ are determined in the following form:
	$$
	\mathcal{J}_{\Phi}(u)=\int_{\mathbb{R}^N}\Phi(u)\,dx,
	$$
	$$
	\mathcal{J}_{s,\Phi}(u)=\int_{\mathbb{R}^N}\int_{\mathbb{R}^N}\Phi\left(\frac{\left|u(x)-u(y)\right|}{|x-y|^s}\right)\frac{dxdy}{|x-y|^N}.
	$$
	Furthermore, we consider the following work space
	$$W_0^{s,\Phi}(\Omega)=\{u\in W^{s,\Phi}(\mathbb{R}^N):u=0\mbox{ in }\Omega^c\}.$$

	Throughout this work we assume the function  $\Phi(t)=\int_{0}^{t}\varphi(\tau)\tau d\tau$ satisfies the following structural conditions:
	\begin{itemize}
		\item [$(\varphi_1)$] It holds that the function $\varphi:\mathbb{R}\setminus \{0\}\to \mathbb{R}$ verifies $t\mapsto t\varphi(t)$ is an odd, increasing homeomorphism from $\; \R\; {\rm into}\; \R$ for each $x,y\in\R^N$;
		\vspace{0,2cm}
		\item [$(\varphi_2)$] $1 \leq \ell\leq \dfrac{t^2\varphi(t)}{\Phi(t)} \leq m<\infty, \quad t\geq0$.
	\end{itemize}
	
	\subsection{Assumptions and main theorems}
	
	The question of regularity has been a central line of research in elliptic PDE since the mid 20th century, with extremely important contributions by Nirenberg, Caffarelli, Krylov, Evans, Figalli, De Giorgi, Nash, Moser, Ladyzhenskaya, Uraltseva, Maz’ya and many others.
	In particular, one of the most famous and important theorems in the theory was developed by De Giorgi \cite{DeGiorgi}. A standard reference for regularity results about fully nonlinear elliptic equations is the book of Caffarelli and Cabre \cite{CX}.

	The main goal of this paper is to extend the previous results to the setting of fractional Orlicz-Sobolev spaces. More specifically, we consider general $N$-functions $\Phi$ which can be more general than powerlike functions $\Phi(t) = |t|^{p}, t \geq 0$ with $p \in (1, N)$. Namely, we shall consider the following nonlinear fractional $\Phi$-Laplacian elliptic problem
	\begin{equation}\label{P}
		\left\{\
		\begin{array}{cl}
			\displaystyle(-\Delta_{\Phi})^{s}u= g(x,u), & \mbox{in}~\Omega,\\
			u=0, & \mbox{on}~\mathbb{R}^N\setminus \Omega,
		\end{array}
		\right.
	\end{equation}
	where $\Omega$ is a open bounded subset in $\mathbb{R}^N,~N\geq 2$, with Lipschitz boundary $\partial \Omega$, $g:\Omega\times\mathbb{R}$ is a Carath\'eodory function satisfying suitable conditions. Moreover,  $\displaystyle(-\Delta_{\Phi})^{s}$ is the nonlocal integro-differential operator of elliptic type defined by
	\begin{eqnarray}\label{operator}
		(-\Delta_{\Phi})^{s}u(x):=2\lim_{\epsilon\searrow 0}\int_{\mathbb{R}^N\setminus B_\epsilon(0)} \varphi\left(D_s u\right)D_su\frac{dy}{|x-y|^{N+s}},
	\end{eqnarray}
	where $ D_s u:=\frac{u(x)-u(y)}{|x-y|^s}$ and $s\in(0,1)$. Under these conditions, we consider the following definition
	\begin{definition}
		A function $u\in W_0^{s,\Phi}(\Omega)$ is said to be a weak solution of \eqref{P} if and only if  
		$$\langle \displaystyle(-\Delta_{\Phi})^s u,v\rangle=\int_{\Omega}g(x,u)vdx,~v\in W_0^{s,\Phi}(\Omega).$$
	\end{definition}
	\noindent Recall that 	$\langle , \rangle$ denotes the usual duality, see for instance \cite{ABSS2}. In order to prove some regularity results for weak solutions for Problem \eqref{P} we assume the following hypothesis:
	\begin{equation}\label{CCG}
		|g(x,t)|\leq a(x)[1+|t|^{\alpha-1}], t \in \mathbb{R}, x \in \Omega,
	\end{equation}
	where $a\in L_{\Psi}(\Omega)$ in such way that $\widetilde{\Psi}\circ\Phi\prec{\Phi}_{s}^*$ with $\alpha=\ell$ or $\alpha=m$. It is worthwhile to mention that \eqref{CCG} says that Problem \eqref{P} admits subcritical or critical behavior. The main idea here is to employ some tools developed in the seminal work \cite{BK}. As a consequence, our main first result can be stated as follows:  
	\begin{theorem}\label{B}
		Suppose that $(\varphi_{1})$, $(\varphi_{2})$, $\ell\in[1,N)$ and $\Phi\in\Delta'$ globally. Assume that \eqref{CCG} holds true. Assume also that one of the following hypothesis is verified:
		\begin{itemize}
			\item[(i)]$\Phi\not\approx t^{m}$ and $\alpha=\ell$ where $m<\ell_s^*$;		
			\item[(ii)] $\Phi\approx t^{m}$ and $\alpha=m$;
		\end{itemize}	
		Then any weak solution $u \in  W_0^{s,\Phi}(\Omega)$ for the problem \eqref{P} belongs to $L^{p}(\Omega)$ for all $ p\in(1,\infty)$.
	\end{theorem}
	
	Now, we shall consider a regularity result by using a different growth condition. Namely, we assume the following conditions
	\begin{eqnarray}
		|g(x,t)|\leq C_1[|t|^{\ell-1}+|t|^{\ell_s^*- (m - \ell)-1}], t \in \mathbb{R}, x \in \Omega,\label{CCG2}\\
		|g(x,t)|\leq C_1[|t|^{m-1}+|t|^{m_s^*-1}], t \in \mathbb{R}, x \in \Omega,\label{CCG3}
	\end{eqnarray}
	where $C_1 > 0$ is a constant. In this way, we can prove the following result:
	
	\begin{corollary}\label{C}
		Suppose that $(\varphi_{1})$, $(\varphi_{2})$, $\ell\in[1,N)$ and $\Phi\in\Delta'$ globally. Assume that \eqref{CCG2}, $\Phi\not\approx t^{m}$ and  $m<\ell_s^*$ hold true. Then any weak solution $u \in  W_0^{s,\Phi}(\Omega)$ for the Problem \eqref{P} belongs to $L^{p}(\Omega), \forall \,\, p\in(1,\infty)$.
	\end{corollary}
	
	\begin{corollary}\label{C1}
		Suppose that $(\varphi_{1})$, $(\varphi_{2})$, $\ell\in[1,N)$ and $\Phi\in\Delta'$ globally. Assume that \eqref{CCG3} and $\Phi\approx t^{m}$ hold true. Then any weak solution $u \in  W_0^{s,\Phi}(\Omega)$ for the Problem \eqref{P} belongs to $L^{p}(\Omega)$ for each  $p\in(1,\infty)$.
	\end{corollary}
	
	Now we shall prove that any solution to the elliptic problem \eqref{P} remains bounded. In order to do that we combine some fine estimates together with the Moser's iteration. Hence, we can state the following result:
	
	\begin{theorem} \label{BK}
		Suppose that $(\varphi_{1})$, $(\varphi_{2})$, $\ell\in[1,N)$ and $\Phi\in\Delta'$ globally. Assume that \eqref{CCG} holds true. Assume also that one of the following hypothesis is verified:
		\begin{itemize}
			\item[(i)]$\Phi\not\approx t^{m}$ and $\alpha=\ell$ where $m<\ell_s^*$;		
			\item[(ii)] $\Phi\approx t^{m}$ and $\alpha=m$;
		\end{itemize}	
		Then any weak solution $u \in  W_0^{s,\Phi}(\Omega)$ for the Problem \eqref{P} is in $L^\infty(\Omega)$.
	\end{theorem}

	\begin{remark}
		It is important to mention that in item $ii)$ of Theorem \ref{B} the conditions $m < \ell^*_s$ is not required. The same statement is also verified for item $ii)$ of Corollary \ref{B} and item $ii)$ of Theorem \ref{BK}.  
	\end{remark}
	
	At this stage, we shall consider a simpler application for Theorem \ref{BK}. More precisely, taking into account the non-homogeneous case $g(x,t)=a(x)\in L_{\Psi}(\Omega)$, we consider the following result: 
	
	\begin{corollary}[The homegeneous case]\label{A}
		Suppose that $(\varphi_{1})$, $(\varphi_{2})$, $\ell\in[1,N)$, $\Phi\in\Delta'$ globally. Assume also that $g(x,t)=a(x)\in L_{\Psi}(\Omega)$ in such way that $\widetilde{\Psi}\circ\Phi\prec\Phi_{s}^*$. Then any weak solution $u \in W_0^{s,\Phi}(\Omega)$ for the Problem \eqref{P} considering the fractional $\Phi$-Laplacian operator is in such way that $ u \in L^{\infty}(\Omega)$. In particular, assuming that $\Phi(t)=|t|^p,~p\in(1,N)$ and $a\in L^q(\Omega)$ with $q\geq \frac{N}{sp}$ we obtain that the weak solutions $u \in W_0^{s,\Phi}(\Omega)$ for the Problem \eqref{P} belongs to $L^{\infty}(\Omega)$.
	\end{corollary}
	
	It is worthwhile to mention that for the $N$-function given by $\Phi(t)=|t|^p,~1<p<\infty$ and $\Psi(t)=|t|^q,~1<q<\infty$ we infer that $(\widetilde{\Psi}\circ\Phi)(t)=|t|^{q'p}$. As a consequence, the condition $\widetilde{\Psi}\circ\Phi<\Phi_s^*$ occurs if and only if $q\geq \frac{N}{sp}$. Therefore, we observe that Corollary \ref{A} generalizes \cite[Theorem 3.1]{Brasco} and \cite{peral}. It is important to mention some prototypes for our main problem. Considering the $N$-function $\Phi$ we obtain the fractional $\Phi$-Laplacian elliptic problem
	\begin{equation}\label{PP}
		\left\{\
		\begin{array}{cl}
			\displaystyle(-\Delta_{\Phi})^{s}u= g(x,u), & \mbox{in}~\Omega,\\
			u=0, & \mbox{on}~\mathbb{R}^N\setminus \Omega.
		\end{array}
		\right.
	\end{equation}
	The fractional $\Phi$-Laplacian operator have been widely studied in the last years, see \cite{Adams1, alves3,rad3, Donaldson,carvalho,FBS}. However, up to  our best knowledge, there exists a few regularity results for Problem \eqref{PP} assuming that $g$ is critical or subcritical, see for the local case \cite{CSA,fu}. A contribution in the present work is to ensure that any weak solution for Problem \eqref{PP} is bounded. For the especial case $\Phi(t) = |t|^p, p \in (1, N)$ the Problem \eqref{PP} became the fractional p-Laplacian elliptic problem given by 
	\begin{equation}\label{PPP}
		\left\{\
		\begin{array}{cl}
			\displaystyle(-\Delta_{p})^{s}u= g(x,u), & \mbox{in}~\Omega,\\
			u=0, & \mbox{on}~\mathbb{R}^N\setminus \Omega.
		\end{array}
		\right.
	\end{equation}
	On this subject we refer the interested reader to \cite{guia} where many results are proved considering existence, multiplicity and nonexistence of weak solutions. Here we refer the reader to the important works \cite{Brasco,guia}. More generality, we consider also the $N$-function $\Phi(t) = |t|^p + |t|^q$ with $1 < p < q < N$ which give us the following elliptic problem 
	\begin{equation}\label{PPPP}
		\left\{\
		\begin{array}{cl}
			\displaystyle(-\Delta_{p})^s + (-\Delta_{q})^{s}u= g(x,u), & \mbox{in}~\Omega,\\
			u=0, & \mbox{on}~\mathbb{R}^N\setminus \Omega.
		\end{array}
		\right.
	\end{equation}
	The Problem \eqref{PPPP} is named as the fractional $(p,q)-$Laplacian elliptic problem. On this topic we infer the reader to \cite{alves} and references therein. For regularity results considering anisotropic functionals we infer the reader to \cite{cianchi}. At this stage, we would like to emphasize that our main results extends the aforementioned works by considering subcritical or critical behavior on $g$ and in view of the generality of the $N$-function $\Phi$. 
	
	It is important to emphasize that for the fractional $\Phi$-Laplacian operator we have some additional difficulties as was pointed before. The main difficulty here is to combine these spaces with the fractional framework. Recall also that $W^{s, \Phi}_0(\Omega)$ and $L^{\Phi}(\Omega)$ are not interpolation spaces. As a consequence, we cannot apply the Moser iteration in a standard way. Namely, we overcame such difficulty using some fine estimates together with the fact that the $N$-function $\Phi$ belongs to the class $\Delta'$. Notice that the Moser's iteration \cite{moser} relies on the classical Lebesgue spaces $L^{p}(\Omega)$ which are not available for $N$-functions. Here we also refer the interested reader to \cite{pucci2,CSA,fang}. The main idea here was control the size of the Orlicz-Gagliardo seminorm by using a specific test function $\theta$, see Lemma \ref{psi-phi} ahead.

	\subsection{Notation} Throughout this work we shall use the following notation:
	
	\begin{itemize}
		\item $C$, $\tilde{C}$, $C_{1}$, $C_{2}$,... denote positive constants (possibly different).
		\item The norm in $L^{p}(\mathbb{R}^N)$ and $L^{\infty}(\mathbb{R}^N)$, will be denoted respectively by $\|\cdot\|_{p}, p \in [1, \infty),$ and $\|\cdot\|_{\infty}$.
		\item The open ball in $\mathbb{R}^N$ centered at $x_0 \in \mathbb{R}^N$ with radii $r > 0$ is denoted by $B_r(x_0)$. Given $A \subset \mathbb{R}^N$ we define $A^c = \{ x \in \mathbb{R}^N : x \notin A \}$.
	\end{itemize}

	\subsection{Outline} The remainder of this work is organized as follows: In the forthcoming section we consider some preliminary results concerning on Orlicz spaces and fractional Orlicz Sobolev spaces. In Section 3 we consider some technical results in order to apply the Moser's iteration for our setting. Section 4 is devoted to the proof of our main results. 
	
	\section{Preliminaries}
	
	\subsection{Orlicz spaces}
	
	In order to allow the fractional $\Phi-$Laplacian operator taking into account the fractional Orlicz space and another cases within the same framework, we introduce some definitions and remarks.
	
	\begin{definition}\label{Musielak function}
		A function $\Phi\colon\R\to\R$ is called an $N$-function if $\Phi$ satisfies the
		following conditions:
		\begin{itemize}
			\item [(a)] $\Phi$ is continuous, convex, increasing and $\Phi(0)=0$.
			\vspace{0,2cm}
			\item [(b)] $\displaystyle\lim_{t\to 0}\frac{\Phi(t)}{t}=0$ and $\displaystyle\lim_{t\to\infty}\frac{\Phi(t)}{t}=\infty$.
		\end{itemize}
	\end{definition}
	
	\begin{remark}
		We point out that assumptions $(\varphi_{1})$ and $(\varphi_{2})$ imply that the class of functions we are considering are $N$-functions.
	\end{remark}
	
	For the function $\Phi$ introduced just above we define the Orlicz class given by 
	$$
	K(\Omega)= \left\{u\colon\Omega\to\R\ \text{measurable}:\ \int_{\Omega}\Phi(|u(x)|)dx<\infty \right\},
	$$
	where $\Omega$ is an open subset in $\mathbb{R}^N,~N\geq 2$, with Lipschitz boundary $\partial \Omega$ and $\Phi$ was defined in \eqref{bar-phi}. The Orlicz space is defined by
	$$
	L_{\Phi}(\Omega)= \left\{u\colon\Omega\to\R\ \text{measurable}:\ \int_{\Omega}\Phi(\lambda|u(x)|)dx<\infty\mbox{ for some }\lambda>0 \right\}.
	$$
	The space $L_{\Phi}(\Omega)$ is a Banach space endowed with the Luxemburg norm
	$$
	\|u\|_{\Phi}=\inf\bigg{\{}\lambda>0:\ \int_{\Omega}\Phi\left(\frac{|u(x)|}{\lambda}\right)dx\leq1\bigg{\}}
	$$
	or the equivalent norm (the Orlicz norm) given by 
	$$
	\|u\|_{(\Phi)}=\sup_{\mathcal{J}_{\widetilde{\Phi}}(v)\leq1}\bigg{|}\int_{\Omega}u(x)v(x)dx\bigg{|}
	$$
	where $\widetilde{\Phi}$ denotes the conjugate Orlicz function of $\Phi$, that is, we put 
	$$
	\widetilde{\Phi}(t):=\sup\{tw -\Phi(w): w>0\}.
	$$
	It is possible to ensure that 
	$$\widetilde{\Phi}(\varphi(t)t)\leq\Phi(2t), \quad t \in \mathbb{R}.$$
	Furthermore, we can consider the following Young inequality
	\begin{equation} \label{Young}
		ab\leq \Phi(a)+\widetilde{\Phi}(b), \quad \text{for all }\ x\in \Omega,\ a,b\geq 0.
	\end{equation}
	As a consequence, we consider the following H\"older's inequality
	$$
	\int_\Omega |uv|\,dx \leq 2\|u\|_{\Phi} \|v\|_{\widetilde{\Phi}},
	$$
	for all $u\in L_{\Phi}(\Omega)$ and $v\in L_{\widetilde{\Phi}}(\Omega)$.
	For the next result we argue as was done in \cite{Fuk_1,Fuk_2}. The main idea is to compare the modular with the Luxemburg  norm. Hence, we consider the following result:
	\begin{proposition}\label{lema_naru}
		Assume that $(\varphi_1)$, $(\varphi_2)$ hold and set
		$$
		\xi_{\Phi}^-(t)=\min\{t^\ell,t^m\}~~\mbox{and}~~ \xi_{\Phi}^+(t)=\max\{t^\ell,t^m\},~~ t\geq 0.
		$$
		Then ${\Phi}$ satisfies the following estimates:
		\begin{eqnarray}\label{des-real}
			\xi_{\Phi}^-(t){\Phi}(\rho)\leq{\Phi}(\rho t)\leq \xi_{\Phi}^+(t){\Phi}(\rho), \quad \rho, t> 0,
		\end{eqnarray}
		$$
		\xi_{\Phi}^-(\|u\|_{{\Phi}})\leq\int_\Omega{\Phi}(u(x))dx\leq \xi_{\Phi}^+(\|u\|_{{\Phi}}), \quad u\in L_{{\Phi}}(\Omega).
		$$
	\end{proposition}
	\begin{remark}
		It is worthwhile to mention that inequality \eqref{des-real} is also verified by using $\Phi^{*}_{s}$ instead of $\Phi$, with $\ell^{*}_{s}$ and $m^{*}_{s}$. More generally, for each $N$-function $\Psi$, we can apply Proposition \ref{lema_naru} by using $\Psi$ instead of the function $\Phi$. In this setting is enough to use $\ell_{\Psi}$ and $m_{\Psi}$ in place of $\ell$ and $m$, respectively.
	\end{remark}
	\begin{definition}
		Let $\Phi$ be an $N$-function. We say that
		\begin{enumerate}
			\item $\Phi$ satisfies the $\Delta_2$-condition globally, in short $\Phi\in \Delta_2$, if there exists $K>0$ such that
			$$\Phi(2t)\leq K \Phi(t),~\forall t\geq 0;$$
			\item $\Phi$ satisfies the $\Delta'$-condition globally, in short $\Phi\in \Delta'$, if there exists $C>0$ such that
			$$\Phi(st)\leq C\Phi(s)\Phi(t),\quad\forall s,t\geq 0.$$
		\end{enumerate}	
	\end{definition}
	It is easy to verify that  the $N$-function $\Phi: \mathbb{R} \to \mathbb{R}$ given by $\Phi(t) = |t|^{p}, t \geq 0$ with $p \in (1, N)$ satisfies the $\Delta'$ condition. More generally, we also consider the $N$-function $\Phi(t) = |t|^p + |t|^q$ with $1 < p < q < N$ as an example of function verifying the $\Delta'$-condition globally. The $N$-function $\Phi(t):=|t|^{p}(|\ln|t||+1)$ satisfies $\Delta'$-condition globally, see \cite[Page 33]{Kras}.  We point that $\Delta'$-condition globally implies that $\Delta_2$-condition globally is satisfied. Assuming that $\Phi$ satisfies $(\varphi_1)$ and $(\varphi_2)$, then $\Phi$ and $\widetilde\Phi$ satisfies $\Delta_2$-condition globally and $L_{\Phi}(\Omega)$ is a reflexive Banach space.
	\begin{definition}
		Let $A,B:\Omega\times \mathbb{R}\to \mathbb{R}$ be two $N$-functions. We say that $A$ is stronger (resp essentially stronger) than $B$, in short we write $A\succ B$ (resp $A\succ \succ B$), if almost every $x\in \overline{\Omega}$ we obtain
		$$B(t)\leq A(at),~t\geq t_0,$$
		for some (respectively for each) $a>0$ and $t_0 > 0$ (respectively depending on $a > 0$).
	\end{definition}
	\begin{remark}
		It is important to emphasize that $A\succ \succ B$ is equivalent to the following condition
		$$\lim_{t\to\infty}\frac{B(\lambda t)}{A(t)}=0,$$
		for all $\lambda>0$.
	\end{remark}

	\subsection{Fractional Orlicz Sobolev spaces}
	
	Now, we shall give a brief overview on the fractional Orlicz-Sobolev. Here we refer the interestead reader to the important works \cite{ABSS1,ABSS2}. Due to the nature of the operator $(-\Delta_{\Phi})^{s}$ defined in \eqref{operator} which is nonlocal, we shall consider the Fractional Orlicz Sobolev spaces $W^{s,\Phi}(\Omega)$ defined by \eqref{W}. This space can be equipped with the norm 
	$$\|u\|_{s,\Phi}:=\|u\|_{\Phi}+[u]_{s,\Phi}, \quad u \in W^{s,\Phi}(\Omega) $$
	where $[.]_{s,\Phi}$ is the Orlicz-Gagliardo seminorm defined by
	$$[u]_{s,\Phi}:=\inf\left\{\lambda>0:\mathcal{J}_{s,\Phi}\left(\frac{u}{\lambda}\right)\leq 1\right\}.$$
	It is important to emphasize that $[.]_{s,\Phi}$ is the Orlicz-Gagliardo seminorm defined by
	$$[u]_{s,\Phi}:=\inf\left\{\lambda>0:\mathcal{J}_{s,\Phi}\left(\frac{u}{\lambda}\right)\leq 1\right\}$$
	where $\Phi$ is an $N$-function. In the special case, taking $\Phi(t) = |t|^p$, the last identity becomes the Gagliardo seminorm for the fractional $p$-Laplacian operator given by
	$$[u]_{s}:=\left\{ \left(\int_{\mathbb{R}^N} \int_{\mathbb{R}^N} \dfrac{|u(x) - u(y)|^p}{|x - y|^{N + sp}} dx dx\right)^{1/p}, u \in W^{s,p}(\Omega), u = 0 \, \, \mbox{in} \,\, \mathbb{R}^N \setminus \Omega \right\}.$$
	\begin{proposition}\label{lema_naru-gagliardo}
		Assume that $(\varphi_1)-(\varphi_2)$ hold. Then the function ${\Phi}$ satisfies the following powerful inequality:
		$$
		\xi_{\Phi}^-([u]_{s,\Phi})\leq\mathcal{J}_{s,\Phi}(u)\leq \xi_{\Phi}^+([u]_{s,\Phi}), \quad u\in W^{s,\Phi}(\Omega),
		$$
		where the functions $\xi_{\Phi}^-, \xi_{\Phi}^+$ was defined in Proposition \ref{lema_naru}.
	\end{proposition}
	\begin{proof}
		The proof follows the same ideas discussed in \cite{Fuk_1,Fuk_2}. We omit the details. 
	\end{proof}
	
	Now, we shall consider some embedding results. To this end, we suppose that $\Phi$ satisfies
	\begin{eqnarray}\label{critical-1}
		\int_0^1\frac{\Phi^{-1}(\tau)}{\tau^{\frac{N+s}{N}}} d\tau<\infty,~\forall x\in\overline{\Omega}
	\end{eqnarray}
	and
	\begin{eqnarray}\label{critical-2}
		\int_1^\infty\frac{\Phi^{-1}(\tau)}{\tau^{\frac{N+s}{N}}} d\tau=\infty,~\forall x\in\overline{\Omega}.
	\end{eqnarray}
	Assuming that \eqref{critical-1} and \eqref{critical-2} are satisfied we define the inverse for the conjugate $N$-function which give the critical behavior using the function $\Phi$. More specifically, we consider the following function:
	$$(\Phi_{s}^*)^{-1}(t)=\int_0^t\frac{\Phi^{-1}(\tau)}{\tau^{\frac{N+s}{N}}} d\tau, \quad  t \geq 0, x \in \Omega.$$
	As a consequence, by using \cite{ABSS2}, we consider the following embedding results: 
	\begin{theorem}
		Let $\Omega$ be a bounded open subset of $\mathbb{R}^N$ and $C^{0,1}-$regularity with bounded boundary. Assume that \eqref{critical-1} and \eqref{critical-2} hold. Then the embedding 
		$$W^{s,\Phi}(\Omega)\hookrightarrow L_{\Phi_{s}^*}(\Omega)$$
		is continuous and 
		$$W^{s,\Phi}(\Omega)\hookrightarrow L_{B}(\Omega)$$
		is compact for all $B\prec\prec \Phi_{s}^*$.
	\end{theorem}
	Now, we introduce a closed linear subspace of $W^{s,\Phi}(\Omega)$ as follows
	$$W^{s,\Phi}_0(\Omega):=\left\{u\in W^{s,\Phi}(\mathbb{R}^N):u=0\mbox{ a.e. in }\mathbb{R}^N\setminus\Omega\right\}.$$
	Furthermore, assuming that $\Omega$ is a bounded open subset of $\mathbb{R}^N$ and $C^{0,1}-$regularity with bounded boundary, we consider the generalized Poincar\'e type inequality \cite{FBS}, i.e., there exists a positive constant $\gamma>0$ such that
	$$
	\|u\|_{\Phi}\leq \gamma [u]_{s,\Phi},~\forall u\in W^{s,\Phi}_0(\Omega).
	$$
	Consequently, we deduce that $\|.\|_{\Phi}$ is a norm in $ W^{s,\Phi}_0(\Omega)$ which is equivalent to the norm $\|.\|_{s,\Phi}$ in $ W^{s,\Phi}_0(\Omega)$.
	
	The next Proposition will be useful to prove Corollary \ref{C}.
	
	\begin{proposition}\label{coro} Assume that $(\varphi_1)-(\varphi_2)$ hold. Then $\Phi^* \circ \Phi^{-1}$ is a $N$-function
		and $\tilde{\Psi} = \Phi^* \circ \Phi^{-1}$ satisfies $m_{\Psi}=(\ell^{*}_{s}/m)^{\prime}$ and $\ell_{\Psi}=(m_{s}^{*}/\ell)^{\prime}$.	
	\end{proposition}
	\begin{proof}
		It is not hard to show that $\Phi^* \circ \Phi^{-1}$ is a N-function by using hypotheses $(\phi_1) - (\phi_2)$. Consider the $N$-function $\tilde{\Psi}=\Phi_{s}^{*}\circ\Phi^{-1}$. Note that
		\[
		1=(\Phi\circ\Phi^{-1})^{\prime}=\Phi^{\prime}(\Phi^{-1}(t))(\Phi^{-1})^{\prime}(t),
		\]
		which implies that
		\[
		(\Phi^{-1}(t))^{\prime}=\dfrac{1}{\Phi^{\prime}(\Phi^{-1}(t))}.
		\]
		Hence, we have
		\[
		\tilde{\Psi}^{\prime}(t)=(\Phi_{s}^{*})^{\prime}(\Phi^{-1}(t))(\Phi^{-1}(t))^{\prime}=(\Phi_{s}^{*})^{\prime}(\Phi^{-1}(t))\dfrac{1}{\Phi^{\prime}(\Phi^{-1}(t))}.
		\]
		Thus, there holds
		\[
		\dfrac{\tilde{\Psi}^{\prime}(t)t}{\tilde{\Psi}(t)}=\dfrac{(\Phi^{*}_{s})^{\prime}(\Phi^{-1}(t))\Phi^{-1}(t)}{\Phi^{*}_{s}(\Phi^{-1}(t))}\dfrac{t}{\Phi^{\prime}(\Phi^{-1}(t))\Phi^{-1}(t)}.
		\]
		By considering $r=\Phi^{-1}(t)$ we obtain
		\[
		\dfrac{\tilde{\Psi}^{\prime}(t)t}{\tilde{\Psi}(t)}=\dfrac{(\Phi^{*}_{s})^{\prime}(r)r}{\Phi^{*}_{s}(r)}\dfrac{\Phi(r)}{\Phi^{\prime}(r)r}.
		\]
		By recalling that
		\[
		\ell_{s}^{\prime}\leq \dfrac{(\Phi^{*}_{s})^{\prime}(r)}{\Phi_{s}^{*}(r)}\leq m_{s}^{*} \quad \mbox{and} \quad \dfrac{1}{m}\leq \dfrac{\Phi(r)}{\Phi^{\prime}(r)r}\leq \dfrac{1}{\ell},
		\]
		we conclude that
		\[
		\ell_{\tilde{\Psi}}:=\dfrac{\ell_{s}^{*}}{m}\leq \dfrac{\tilde{\Psi}(t)t}{\tilde{\Psi}(t)}\leq\dfrac{m_{s}^{*}}{\ell}=:m_{\tilde{\Psi}}.
		\]
		As a consequence, $m_{\Psi} = \left(\ell_{s}^{*}/m\right)'$ and $\ell_{\Psi} = \left(m_{s}^{*}/\ell\right)'$. This ends the proof. 
	\end{proof}
	
	\section{Some technical results}
	
	In this section we shall prove some technical results related to the fractional Orlicz spaces. In the first step, we consider the following set 
	\begin{equation}
		T = \{\theta \in C^{0}(\mathbb{R}, \mathbb{R}); \theta' \,\, \mbox{exists in the set} \,\, \mathbb{R}\setminus \Gamma_{\theta} \,\, \mbox{where} \,\, \Gamma_{\theta} \,\, \mbox{has at most a finite number of elements}\}.
	\end{equation}
	
	\begin{lemma}\label{Japlication-1}
		Let $\Phi$ be an $N$-function such that $\Phi\in \Delta'$ and $0<s<1$. Define the auxiliary function
		$$\Gamma_\theta(t):=\int_0^t\Phi^{-1}(\theta'(\tau))d\tau, t \geq 0,$$
		where $\theta\in T$ is a nondecreasing function. Then there exists $C>0$ such that
		$$C\Phi\left(\frac{\Gamma_\theta(b)-\Gamma_\theta(a)}{|x-y|^s}\right)\leq \left[\frac{\theta(b)-\theta(a)}{|x-y|^s}\right]\varphi\left(\frac{b-a}{|x-y|^s}\right)\left(\frac{b-a}{|x-y|^s}\right),$$
		for all $a,b\in\mathbb{R},~x\neq y\in \mathbb{R}^N$.
	\end{lemma}
	\begin{proof}
		Since $\Phi\in \Delta'$, there exists $C>0$ such that
		\begin{eqnarray}
			C\frac{\Phi\left(\frac{\Gamma_\theta(b)-\Gamma_\theta(a)}{|x-y|^s}\right)}{\Phi\left(\frac{b-a}{|x-y|^s}\right)}&\leq&\Phi\left(\frac{\frac{\int_a^b\Phi^{-1}(\theta'(\tau))d\tau}{|x-y|^s}}{\frac{b-a}{|x-y|^s}}\right)\leq\Phi\left(\frac{\int_a^b\Phi^{-1}(\theta'(\tau))d\tau}{b-a}\right).\nonumber
		\end{eqnarray}
		Now, by using Jensen's inequality, we infer that
		$$C\frac{\Phi\left(\frac{\Gamma_\theta(b)-\Gamma_\theta(a)}{|x-y|^s}\right)}{\Phi\left(\frac{b-a}{|x-y|^s}\right)}\leq\frac{\displaystyle\int_a^b\Phi\left(\Phi^{-1}(\theta'(\tau))\right)d\tau}{b-a}=\frac{\theta(b)-\theta(a)}{b-a}.$$
		It follows from the inequality $\Phi(t)\leq \varphi(t)t^2$ that
		\begin{eqnarray}
			C\Phi\left(\frac{\Gamma_\theta(b)-\Gamma_\theta(a)}{|x-y|^s}\right)&\leq& \frac{\theta(b)-\theta(a)}{b-a}\Phi\left(\frac{b-a}{|x-y|^s}\right)\nonumber\\
			&\leq& \left[\frac{\theta(b)-\theta(a)}{|x-y|^s}\right]\varphi\left(\frac{b-a}{|x-y|^s}\right)\left(\frac{b-a}{|x-y|^s}\right).\nonumber
		\end{eqnarray}
		This ends the proof. 	
	\end{proof}
	In what follows, given any $N$-function $M$ satisfying
	$$1\leq \ell_M<\frac{tM'(t)}{M(t)}\leq m_M<\infty,~t>0,$$
	for some positive constants $\ell_M$ and $m_M$, we fix some notations
	$$\xi_{M}^{-}(t):=\min\{t^{\ell_M},t^{m_M}\}\qquad\mbox{and}\qquad \xi_{M}^{+}(t):=\max\{t^{\ell_M},t^{m_M}\}.$$
	Moreover, for each $L,\beta>1$, we define
	$$\theta(t):=t\Phi\left(\frac{t_L^{\beta-1}}{A}\right)\qquad\mbox{and}\qquad \Gamma_{\theta}(t):=\int_0^t\Phi^{-1}(\theta'(\tau))d\tau$$
	where $t_L:=\min\{t,L \}$ and $A$ shall be defined later. The main idea here is to control the size of A in order to apply the Moser's iteration. As a first step, by using some fine estimates together with the claim rule, we obtain the following result:
	\begin{lemma}\label{estGamma}
		The function $\Gamma_{\theta}$ satisfies the following estimate
		$$\Gamma_{\theta}(t)\geq \frac{1}{\beta A}tt_L^{\beta-1}, t \geq 0.$$
	\end{lemma}
	\begin{proof}
		It is not hard to see that 
		\begin{eqnarray}\label{der-theta}
			\theta'(t)=\Phi\left(\frac{t^{\beta-1}_L}{A}\right)+(\beta-1)\Phi'\left(\frac{t^{\beta-1}_L}{A}\right)\frac{tt_L^{\beta-2}}{A}(t_L)',~t\neq L.
		\end{eqnarray}
		It follows also from $(\varphi_2)$ that
		$$\theta'(t)\geq\beta\Phi\left(\frac{t^{\beta-1}_L}{A}\right),~\mbox{if}~t<L\qquad\mbox{and}\qquad \theta'(t)=\Phi\left(\frac{t^{\beta-1}_L}{A}\right),~\mbox{if}~t>L.$$
		More generally, we deduce that 
		\begin{equation}\label{des-theta}
			\Phi\left(\frac{t^{\beta-1}_L}{A}\right)\leq \theta'(t)\leq [1+(\beta-1)m]\Phi\left(\frac{L^{\beta-1}}{A}\right),~t\neq L.
		\end{equation}
		Consequently,
		$$\Gamma_{\theta}(u)\geq \frac{tt_L^{\beta-1}}{A},~\mbox{if}~t> L\qquad\mbox{and}\qquad \Gamma_{\theta}(u)\geq \frac{1}{A\beta}tt_L^{\beta-1},~\mbox{if}~t\leq L.$$
		As a consequence, we obtain that $\Gamma_{\theta}(u)\geq tt_L^{\beta-1}/(\beta A)$. This finishes the proof.	
	\end{proof}
	
	\begin{lemma}\label{psi-phi}
		Assume that $\widetilde{\Psi}\circ\Phi\prec\Phi_{s}^*$. Let $C^* > 0$ be fixed in such way that
		\begin{eqnarray}\label{emb-optimal}
			\|v\|_{\widetilde{\Psi}\circ\Phi}\leq C^*\|v\|_{\Phi_{s}^*},~\forall v\in W_0^{s,\Phi}(\Omega).
		\end{eqnarray}
		Then
		\begin{eqnarray}\label{Luxnorm}
			\left\|\Phi\left(\frac{v}{C^*\|v\|_{\Phi_{s}^*}}\right)\right\|_{\widetilde{\Psi}}\leq 1,~\forall v\in W_0^{s,\Phi}(\Omega).
		\end{eqnarray}
	\end{lemma}
	\begin{proof}
		Indeed, by using \eqref{emb-optimal}, we deduce that 
		\begin{eqnarray}
			\int_{\Omega}\widetilde{\Psi}\left(\Phi\left(\frac{v}{C^*\|v\|_{\Phi_{s}^{*}}}\right)\right)dx\leq \int_{\Omega}\widetilde{\Psi}\circ\Phi\left(\frac{v}{\|v\|_{\widetilde{\Psi}\circ\Phi}}\right)dx=1.\nonumber
		\end{eqnarray}
		Now, taking into account the Luxemburg's norm, we obtain that \eqref{Luxnorm} is now verified. This ends the proof.
	\end{proof}
	
	In the next result we shall use the Clarke's derivative. We refer the reader to \cite{chang} for more details around this subject. More specifically, we prove the following result
	\begin{lemma}\label{test}
		Let $u\in W_0^{s,\Phi}(\Omega)$ be fixed. Then the function $\theta(u)\in W_0^{s,\Phi}(\Omega)$.
	\end{lemma}
	\begin{proof}
		Our goal is to use Lebourg's Mean Value Theorem \cite{chang}. Assuming that $t\neq L$, $\theta'(t)$ is given by \eqref{der-theta}. If $t=L$, the Clarke derivate of $\theta$ in $t=L$ is given by
		\begin{equation}\label{des-theta-2}
			\partial  \theta(L)=\left[\Phi\left(\frac{L^{\beta-1}}{A}\right),\Phi\left(\frac{L^{\beta-1}}{A}\right)+(\beta-1)\Phi'\left(\frac{L^{\beta-1}}{A}\right)\frac{L^{\beta-1}}{A}\right].
		\end{equation}
		Using \eqref{des-theta} and \eqref{des-theta-2}, we deduce that
		$$|\eta|\leq [1+(\beta-1)m ]\Phi\left(\frac{L^{\beta-1}}{A}\right),~\forall \eta\in \partial  \theta(t).$$
		It follows from Lebourg's Mean Value Theorem \cite{chang} and Proposition \ref{lema_naru} that
		\begin{eqnarray}\label{des-lebourg}
			|\theta(a)-\theta(b)|&\leq&\left|\sup_{t\in[0,1]}\partial\theta((1-t)a+tb)\right||a-b|\nonumber\\
			&\leq& [1+(\beta-1)m ]\Phi(1)\xi_{\Phi}^+\left(\frac{L^{\beta-1}}{A}\right)|a-b|,
		\end{eqnarray}
		holds true for each $a,b\in\mathbb{R}$. According to \eqref{des-lebourg} we also obtain
		$$\int_{\Omega}\int_{\Omega}\Phi\left(|D_s\theta(u)|\right)d\mu\leq \xi_{\Phi}^{+}\left([1+(\beta-1)m ]\Phi(1)\xi_{\Phi}^+\left(\frac{L^{\beta-1}}{A}\right)\right)\int_{\Omega}\int_{\Omega}\Phi\left(|D_su|\right)d\mu<\infty.$$
		This  finishes the proof.
	\end{proof}
	\section{Proof of the main results}
	
	\noindent{\bf Proof of Theorem \ref{B}:} According to Lemma \ref{test}, we take  $\theta(u)$ as a test function for our main problem \eqref{P}. As a product, we obtain
	\begin{eqnarray}\label{eq11}
		\int_{\R^N}\int_{\R^N}\varphi(|D_su|)D_suD_s(\theta(u))d\mu=\int_{\Omega}g(x,u) \theta(u)dx.
	\end{eqnarray}
	Now, using Proposition \ref{lema_naru} and Lemma \ref{Japlication-1}, we also infer that 
	\begin{eqnarray}\label{eq21}
		C&=&C\int_{\Omega}\int_{\Omega}\Phi\left(\frac{|D_s \Gamma_{\theta} (u)|}{[\Gamma_{\theta}(u)]_{s,\Phi}}\right)d\mu \leq\xi_{\Phi}^+\left(\frac{1}{[\Gamma_{\theta}(u)]_{s,\Phi}}\right)C\int_{\Omega}\int_{\Omega}\Phi\left(|D_s \Gamma_{\theta} (u)|\right)d\mu\nonumber\\
		&\leq& \xi_{\Phi}^+\left(\frac{1}{[\Gamma_{\theta}(u)]_{s,\Phi}}\right)\int_{\Omega}\int_{\Omega}\varphi(|D_su|)D_suD_s(\theta(u))d\mu,
	\end{eqnarray}
	where $D_s(\theta(u))=\frac{\theta(u(x))-\theta (u(y))}{|x-y|^s}$. Since $W_0^{s,\Phi}(\Omega)\hookrightarrow L_{\Phi_{s}^{*}}(\Omega)$, there exists $S_{s,\Phi}>0$ such that $S_{s,\Phi}\|v\|_{\Phi_{s}^*}\leq [v]_{s,\Phi},~v\in W_0^{s,\Phi}(\Omega)$. Using the last assertion together with \eqref{eq11} and \eqref{eq21} we obtain that 
	\begin{equation}\label{eq31}
		C\leq \xi_\Phi^+\left(\frac{S_{s,\Phi}}{\|\Gamma_{\theta}(u)\|_{\Phi_{s}^{*}}}\right)\int_{\Omega}g(x,u)\theta(u)dx.
	\end{equation}
	In view of Lemma \ref{estGamma} we deduce that 
	$\Gamma_{\theta}(\tau)\geq \frac{1}{A\beta}\tau\tau_L^{\beta-1}$, for $\tau\geq 0$. Thus, by using \eqref{eq31}, we obtain
	\begin{eqnarray}\label{des11a}
		C&\leq&\xi_\Phi^+\left(\frac{\beta AS_{s,\Phi}}{\|uu_L^{\beta-1}\|_{\Phi_{s}^{*}}}\right)\int_{\Omega}g(x,u)\theta(u)dx.
	\end{eqnarray}
	On the other hand, fixing $R>0$, it follows from \eqref{CCG} and Proposition \ref{lema_naru} that
	\begin{eqnarray}\label{desBK1}
		\int_{\Omega}g(x,u)\theta(u)dx&\leq&  \left(\int_{[|u|<1]}+\int_{[|u|\geq 1]}\right)a(x)\left[1+|u|^{\ell-1}\right]u\Phi\left(\frac{u_L^{\beta-1}}{A}\right)dx \nonumber\\
		&\leq& 2\|a\|_1\Phi(1)\xi_\Phi^+\left(\frac{1}{A}\right)+ \int_{[|u|\geq1]}a(x)\xi_\Phi^-(|u|)\Phi\left(\frac{u_L^{\beta-1}}{A}\right)dx \nonumber\\
		&\leq& \!\! \left(\int_{[|u|\geq1]\cap[a\leq R]}+\int_{[|u|\geq1]\cap[a> R]}\right)a(x)\Phi\left(\frac{uu_L^{\beta-1}}{A}\right)dx + 2\|a\|_1\Phi(1)\xi_\Phi^+\left(\frac{1}{A}\right).
	\end{eqnarray}
	Thanks to \eqref{des11a}, \eqref{desBK1} and using H\"{o}lder's inequality we obtain that
	\begin{eqnarray}\label{desBK3}
		C&\leq & \xi_\Phi^+\left(\frac{\beta AS_{s,\Phi}}{\|uu_L^{\beta-1}\|_{\Phi_{s}^{*}}}\right)\left[2\|a\|_1\Phi(1)\xi_\Phi^+\left(\frac{1}{A}\right)+R\int_{[|u|\geq1]}\Phi\left(\frac{uu_L^{\beta-1}}{A}\right)dx\right.\nonumber\\
		&+&\left.2\|a\chi_{[a>R]}\|_{\Psi}\left\|\Phi\left(\frac{uu_L^{\beta-1}}{A}\right)\right\|_{\widetilde{\Psi}} \right].
	\end{eqnarray}
	Under these conditions, we take $A=C^*\|u u_L^{\beta-1}\|_{\Phi_{s}^*}$. As a consequence, by using Lemma \ref{psi-phi} and Lemma \ref{lema_naru}, it follows from \eqref{desBK3} that
	\begin{eqnarray}\label{desBK4}
		C&\leq & \xi_\Phi^+\left(\frac{\beta C^*\|u u_L^{\beta-1}\|_{\Phi_{s}^*}S_{s,\Phi}}{\|uu_L^{\beta-1}\|_{\Phi_{s}^{*}}}\right)\left[2\|a\|_1\Phi(1)\xi_\Phi^+\left(\frac{1}{C^*\|u u_L^{\beta-1}\|_{\Phi_{s}^*}}\right)\right.\nonumber\\
		&+&\left.R\int_{[|u|\geq1]}\Phi\left(\frac{uu_L^{\beta-1}}{C^*\|u u_L^{\beta-1}\|_{\Phi_{s}^*}}\right)dx+2\|a\chi_{[a>R]}\|_{\Psi} \right]\nonumber\\
		&\leq&\xi_\Phi^+\left(S_{s,\Phi}C^*\right)\beta^m\left[2\|a\|_1\Phi(1)\xi_\Phi^+\left(\frac{1}{C^*\|u u_L^{\beta-1}\|_{\Phi_{s}^*}}\right)\right.\nonumber\\
		&+&\left.R\xi_\Phi^+\left(\frac{1}{C^*\|u u_L^{\beta-1}\|_{\Phi_{s}^*}}\right)\int_{[|u|\geq1]}\Phi\left(uu_L^{\beta-1}\right)dx+2\|a\chi_{[a>R]}\|_{\Psi} \right]\nonumber\\
		&\leq&\xi_\Phi^+\left(S_{s,\Phi}C^*\right)\beta^m\left[\frac{2\|a\|_1\Phi(1)}{\xi_\Phi^-\left(C^*\|u u_L^{\beta-1}\|_{\Phi_{s}^*}\right)}+2\|a\chi_{[a>R]}\|_{\Psi}\right.\nonumber\\
		&+&\left.R\frac{\Phi(1)}{\xi_\Phi^-\left(C^*\|u u_L^{\beta-1}\|_{\Phi_{s}^*}\right)}\int_{[|u|\geq1]}\left(uu_L^{\beta-1}\right)^mdx \right].
	\end{eqnarray}
	Furthermore, choosing $R=R(\beta)>0$ in such way that
	$$2\xi_\Phi^+\left( S_{s,\Phi}C^*\right)\beta^m\|a\chi_{[a> R]}\|_{\Psi} \leq \frac{C}{2}$$
	and using the embedding $L_{\Phi_{s}^*}(\Omega)\hookrightarrow L^{\ell_s^*}(\Omega)$ and $0\leq u_L\leq u$ which together with \eqref{desBK4} imply that 
	\begin{eqnarray}\label{desBK42}
		\xi_\Phi^-\left(\|u u_L^{\beta-1}\|_{\ell_s^*}\right)&\leq & D\beta^m\left[1+R\|u\|_{\beta m}^{\beta m}\right].
	\end{eqnarray}
	for some positive constant $D=D(\Phi,s,a, \Psi)$.
	
	Now, we shall consider $\beta=\beta_1:=\frac{\ell_s^*}{m}>1$. Thus,
	$$\xi_\Phi^-\left(\|u u_L^{\frac{\ell_s^*}{m}-1}\|_{\ell_s^*}\right)\leq  D\beta^m\left[1+R\|u\|_{\ell_s^*}^{\ell_s^*}\right].$$
	Taking the limit as $L\to \infty$ we obtain that  $u\in L^{\ell_s^* \beta_1}(\Omega)$. Here, it is important to point that $\frac{(\ell_s^*)^2}{m}=\ell_s^*\beta_1>\ell_s^*$. More generally, for each $k\geq 1$, we consider
	$$\beta_{k+1}m=\beta_k\ell_s^*, \quad \mbox{where } \beta_1=\frac{\ell_s^*}{m}.$$
	In this way, assuming that $\delta:=\frac{\ell_s^*}{m}$, we infer that $\beta_{k}=\delta^k$. In view of \eqref{desBK42} we mention that
	\begin{eqnarray}\label{desBK5}
		\xi_\Phi^-(\|u\|_{{\beta_{k+1}}\ell_s^{*}}^{\beta_{k+1}})\leq D\beta_{k+1}^m\left(1+R\|u\|_{\beta_k \ell^*_s}^{\beta_k \ell^*_s}\right).
	\end{eqnarray}
	It follows from \eqref{desBK5} that $u\in L^{\delta^k}(\Omega), \, \forall \, k\in\mathbb{N}$. Since $\beta_k\to\infty$ as $k\to \infty$, for each $p\in(1,\infty)$, there exists $k_0\in \mathbb{N}$ such that $p<\beta_{k_0}=\delta^{k_0}$. This implies that $u\in L^p(\Omega)$. This finishes the proof of Theorem \ref{B}-i).
	\vskip.3cm
	\noindent{\bf Proof of Theorem \ref{B}-ii):} Here, $\Phi$ is equivalent to $t\mapsto |t|^m$. In this case, $W_0^{s,\Phi}(\Omega)=W_0^{s,m}(\Omega)\hookrightarrow L^{m_s^*}(\Omega)$. Consequently, the inequality \eqref{desBK42} can be changed in the following form
	\begin{eqnarray}\label{desBK42ii}
		\xi_\Phi^-\left(\|u u_L^{\beta-1}\|_{m_s^*}\right)&\leq & D\beta^m\left[1+R\|u\|_{\beta m}^{\beta m}\right].
	\end{eqnarray}
	Moreover, we consider
	$$\beta_{k+1}m=\beta_km_s^*, \quad \mbox{where } \beta_1=\frac{m_s^*}{m}=\frac{N}{N-sm}.$$
	As a consequence, we deduce that
	\begin{eqnarray}\label{desBK5ii}
		\xi_\Phi^-(\|u\|_{{\beta_{k+1}}m_s^{*}}^{\beta_{k+1}})\leq D\beta_{k+1}^m\left(1+R\|u\|_{\beta_k m^*_s}^{\beta_k m^*_s}\right).
	\end{eqnarray}
	From now on, by using the same steps discussed in the proof of Theorem \ref{B}-i) we obtain that $u \in L^p(\Omega)$ for each $p \in (1, \infty)$. This ends the proof.
	\vskip.3cm
	\noindent{\bf Proof of Corollary \ref{C}:}
	Let $u \in  W_0^{s,\Phi}(\Omega)$ be any weak solution for Problem \eqref{P}.
	Firstly, we consider the following auxiliary function
	\begin{equation}\label{salva}
		a(x):=\frac{C[|u(x)|^{\ell-1}+|u(x)|^{\ell_s^*-(m - \ell)-1}]}{[1+|u(x)|^{\ell-1}]}, x \in \Omega.
	\end{equation}
	It is not hard to see that $\Phi_s^* \circ \Phi^{-1}$ is also a $N$-function. Now, consider $\Psi = (\Phi_s^* \circ \Phi^{-1})\widetilde{}$ which implies that $\widetilde{\Psi}\circ \Phi = \Phi_s^*$. In particular, we know that $m_{\Psi} = (\ell_s^*/m)'$ and $l_{\Psi} = (m_s^*/\ell)'$, see Proposition \ref{coro}. The main idea here is to apply Theorem \ref{B} which requires that the function $a \in L_{\Psi}(\Omega)$. In order to do that, by using \eqref{salva} and Proposition \ref{lema_naru}, we obtain that
	$$\int_{\Omega} \Psi(a) dx\leq C_1 + C_2 \int_{|u| \geq 1} |u|^{(\ell^*_s - m)(l_s^*/m)'}dx \leq C_1 + C_2\int_{\Omega} |u|^{\ell_s^*} dx <\infty$$
	holds for some constants $C_1 > 0$ and $C_2 > 0$. As a consequence, we deduce that 
	$$|g(x,t)| \leq b(x)[1 + |u(x)|^{l-1}], x \in \Omega, t \in \mathbb{R}$$
	holds true with $b \in L_{\Psi}(\Omega)$ due to the fact that $b(x) = C a(x), x \in \Omega$ for some $C > 0$. Hence, taking into account Theorem \ref{B}, the solution $u$ is in $L^{p}(\Omega)$  for each $p\in(1,\infty)$. This finishes the proof.

	\noindent{\bf Proof of Theorem \ref{BK}:} The proof here is similar to the proof of Theorem \ref{B}. In fact, by using \eqref{desBK3}, we obtain
	\begin{eqnarray}\label{desBK3b}
		C&\leq & \xi_\Phi^+\left(\beta C^*S_{s,\Phi}\right)\left[2\|a\|_1\Phi(1)\xi_\Phi^+\left(\frac{1}{C^*\|u u_L^{\beta-1}\|_{\Phi_{s}^*}}\right)\right.\nonumber\\
		&+&R\int_{[|u|\geq1]}\Phi\left(\frac{uu_L^{\beta-1}}{C^*\|u u_L^{\beta-1}\|_{\Phi_{s}^*}}\right)dx +\left.2\|a\chi_{[a>R]}\|_{\Psi}\left\|\Phi\left(\frac{uu_L^{\beta-1}}{C^*\|u u_L^{\beta-1}\|_{\Phi_{s}^*}}\right)\right\|_{\widetilde{\Psi}} \right].
	\end{eqnarray}
	Now, taking into account some differences comparing to the proof of Theorem \ref{B}, we shall give some details. It follows from \eqref{desBK3b} and Lemma \ref{psi-phi}, Proposition \ref{lema_naru} that
	\begin{eqnarray}\label{desBK4b}
		C&\leq & \xi_\Phi^+\left(\beta C^*S_{s,\Phi}\right)\left[\frac{2\|a\|_1\Phi(1)}{\xi_\Phi^-\left(C^*\|u u_L^{\beta-1}\|_{\Phi_{s}^*}\right)}\right. +\left.R\int_{[|u|\geq1]}\Phi\left(\frac{uu_L^{\beta-1}}{\|uu_L^{\beta-1}\|_{\Phi}}\frac{\|uu_L^{\beta-1}\|_{\Phi}}{C^*\|u u_L^{\beta-1}\|_{\Phi_{s}^*}}\right)dx+2\|a\chi_{[a>R]}\|_{\Psi}\right]\nonumber\\
		&\leq & \xi_\Phi^+\left(\beta C^*S_{s,\Phi}\right)\left[\frac{2\|a\|_1\Phi(1)}{\xi_\Phi^-\left(C^*\|u u_L^{\beta-1}\|_{\Phi_{s}^*}\right)}\right.\nonumber\\
		&+&\left.R\xi_\Phi^+\left(\frac{S_\Phi^*\|uu_L^{\beta-1}\|_{\Phi}}{C^*S_\Phi^*\|u u_L^{\beta-1}\|_{\Phi_{s}^*}}\right)\int_{[|u|\geq1]}\Phi\left(\frac{uu_L^{\beta-1}}{\|uu_L^{\beta-1}\|_{\Phi}}\right)dx+2\|a\chi_{[a>R]}\|_{\Psi}\right],\nonumber\\
	\end{eqnarray}
	where $\displaystyle S_\Phi^*:=\inf_{v\in L_{\Phi_{s}^*}(\Omega)\setminus\{0\}}\frac{\|v\|_{\Phi_{s}^*}}{\|v\|_{\Phi}}$. Under these conditions, we also deduce that 
	$$\frac{S_\Phi^*\|uu_L^{\beta-1}\|_{\Phi}}{\|u u_L^{\beta-1}\|_{\Phi_{s}^*}}\leq 1.$$
	Therefore, by using Proposition \ref{lema_naru} and \eqref{desBK4b}, we infer that 
	\begin{eqnarray}\label{desBK4c}
		C&\leq & \xi_\Phi^+\left(\beta C^*S_{s,\Phi}\right)\left[\frac{2\|a\|_1\Phi(1)}{\xi_\Phi^-(C^*)\xi_\Phi^-\left(\|u u_L^{\beta-1}\|_{\Phi_s^*}\right)}+2\|a\chi_{[a>R]}\|_{\Psi}\right.+\left.\frac{R}{\xi_\Phi^-\left(C^*S_\Phi^*\right)}\left(\frac{S_\Phi^*\|uu_L^{\beta-1}\|_{\Phi}}{\|u u_L^{\beta-1}\|_{\Phi_{s}^*}}\right)^\ell\right].
	\end{eqnarray}
	Here, it was used the fact that
	$$\int_{\Omega}\Phi\left(\frac{uu_L^{\beta-1}}{\|uu_L^{\beta-1}\|_{\Phi}}\right)dx=1.$$
	Now, assuming that $\|u u_L^{\beta-1}\|_{\Phi_{s}^*}\geq 1$, it follows from \eqref {desBK4c} that
	\begin{eqnarray}\label{desBK4d}
		C\|u u_L^{\beta-1}\|_{\Phi_{s}^*}^\ell&\leq & \xi_\Phi^+\left(\beta C^*S_{s,\Phi}\right)\left[\frac{2\|a\|_1\Phi(1)}{\xi_\Phi^-(C^*)}+\frac{R(S_\Phi^*)^\ell}{\xi_\Phi^-\left(C^*S_\Phi^*\right)}\|uu_L^{\beta-1}\|_{\Phi}^\ell\right.\nonumber\\
		&+&\left.2\|a\chi_{[a>R]}\|_{\Psi}\|u u_L^{\beta-1}\|_{\Phi_{s}^*}^\ell\right].
	\end{eqnarray}
	Under these conditions, choosing $R=R(\beta)>0$ in such way that
	$$\xi_\Phi^+\left( S_{s,\Phi}C^*\beta\right)\|a\chi_{[a> R]}\|_{\Psi} \leq \frac{C}{2}$$
	and taking into account \eqref{desBK4d} we deduce that
	\begin{eqnarray}\label{desBK4e}
		\|u u_L^{\beta-1}\|_{\Phi_{s}^*}^\ell&\leq &D\beta^m\left[1+R\|uu_L^{\beta-1}\|_{\Phi}^\ell\right],
	\end{eqnarray}
	where $D= \frac{2\xi_\Phi^+\left( C^*S_{s,\Phi}\right)}{C}\max\left\{\frac{2\|a\|_1\Phi(1)}{\xi_\Phi^-(C^*)},\frac{(S_\Phi^*)^\ell}{\xi_\Phi^-\left(C^*S_\Phi^*\right)}\right\}$.
	
	On the other hand, assuming that $\|u u_L^{\beta-1}\|_{\Phi_{s}^*}< 1$, it follows from \eqref {desBK4c} that
	\begin{eqnarray}\label{desBK4f}
		\|u u_L^{\beta-1}\|_{\Phi_{s}^*}^m&\leq & \xi_\Phi^+\left(\beta C^*S_{s,\Phi}\right)\left[\frac{2\|a\|_1\Phi(1)}{\xi_\Phi^-(C^*)}+2\|a\chi_{[a>R]}\|_{\Psi}\|u u_L^{\beta-1}\|_{\Phi_{s}^*}^{m}\right.\nonumber\\
		&+&\left.\frac{R(S_\Phi^*)^\ell}{\xi_\Phi^-\left(C^*S_\Phi^*\right)}\|uu_L^{\beta-1}\|_{\Phi}^\ell\|u u_L^{\beta-1}\|_{\Phi_{s}^*}^{m-\ell}\right].\nonumber\\
		&\leq& D\beta^m\left[1+R\|uu_L^{\beta-1}\|_{\Phi}^\ell\right].
	\end{eqnarray}
	According to \eqref{desBK4e} and \eqref{desBK4f} we also infer that 
	\begin{eqnarray}\label{desBK4g}
		\xi_\Phi^-\left(\|u u_L^{\beta-1}\|_{\Phi_{s}^*}\right)&\leq &D\beta^m\left[1+R\|uu_L^{\beta-1}\|_{\Phi}^\ell\right].
	\end{eqnarray}
	Furthermore, by using Theorem \ref{B}, we ensure that $u\in L^{p}(\Omega),~p\in[1,\infty)$. Recall also that $0\leq u_L\leq u$ holds true. As a consequence, we obtain that
	$uu_L^{\beta-1}\leq u^\beta\in L^{m_s^*}(\Omega)\hookrightarrow L_{\Phi_{s}^*}(\Omega)$. Therefore, we take the limit as $L$ goes to infinity in the expression given in \eqref{desBK4g} obtaining the following estimate
	\begin{eqnarray}\label{desBK4h}
		\xi_\Phi^-\left(\|u^\beta\|_{\Phi_{s}^*}\right)&\leq &D\beta^m\left[1+R\|u^\beta\|_{\Phi}^\ell\right].
	\end{eqnarray}
	Notice also that the embeddings $L_{\Phi_{s}^*}(\Omega)\hookrightarrow L_{\ell_s^*}(\Omega)$ and $L^m(\Omega)\hookrightarrow L_{\Phi}(\Omega)$ are essential in our arguments. In fact, using the embeddings listed just above, the estimate \eqref{desBK4h} can be rewritten in the following form
	\begin{eqnarray}\label{desBK4i}
		\xi_\Phi^-\left(\|u\|_{\beta\ell_s^*}^\beta\right)&\leq &\overline{D}\beta^m\left[1+R\|u\|_{m\beta}^{\ell\beta}\right],
	\end{eqnarray}
	for some $\overline{D}>0$. In general, for each $k\geq 1$, we consider
	$$\beta_{k+1}m=\beta_k\ell_s^*, \quad \mbox{where } \beta_1=\frac{\ell_s^*}{m}.$$
	It is not hard  to deduce that $\beta_k=\left(\ell_s^*/m\right)^k$. In what follows we consider  $F_{k+1}=\beta_{k+1}\ln(\|u\|_{\beta_{k+1}\ell_s^*})$.
	At this stage, we consider the following set
	$$N_+(u):=\{k\in\mathbb{N}|\|u\|_{m\beta_k}\geq 1\}.$$
	In light of \eqref{desBK4i} we also obtain
	\begin{eqnarray}\label{desprint21}
		F_{k+1} &\leq& \ln (\overline{D}\beta_{k+1}^m)^{\frac{1}{\alpha}}+\frac{1}{\alpha}\ln\left(1+R\|u\|_{m\beta_k}^{\ell \beta_k}\right)\leq \ln (2R(\overline{D} \beta_{k+1}^m)^{\frac{1}{\alpha}})+F_k\nonumber\\
		&\leq& \ln (2R(\overline{D} \beta_{k+1}^m)^{\frac{1}{\alpha}})+\beta_1F_k,~k\in N_+(u)
	\end{eqnarray}
	and
	\begin{eqnarray}\label{desprint21a}
		F_{k+1} &\leq&  \ln(2R(\overline{D} \beta_{k+1}^m)^{\frac{1}{\alpha}}), \quad k\in\mathbb{N}\setminus N_+(u),
	\end{eqnarray}
	where $\alpha\in\{\ell,m\}$. It follows from \eqref{desprint21} and \eqref{desprint21a} that
	\begin{equation}\label{desprint31}
		F_{k+1}\leq
		\left\{
		\begin{array}{ll}
			\lambda_k+\beta_1F_k, & \mbox{if }k\in N_+(u),\\
			\lambda_k, & \mbox{if }k\in \mathbb{N}\setminus N_+(u),
		\end{array}
		\right.
	\end{equation}
	where $\lambda_k=\ln(2R(\overline{D} \beta_{k+1}^m)^{\frac{1}{\alpha}})$. Furthermore, we take $d=\ln(2R\overline{D}^\frac{1}{\alpha})$ and $\hat{d}=\frac{m}{\alpha}\ln\left(\ell_s^m/m\right)$, where $\alpha\in\{\ell,m\}$. As a consequence, we obtain
	\begin{equation}\label{desprint41}
		\frac{\lambda_k}{\beta_1^k}\leq \frac{d}{\beta_1^k}+\hat{d}\frac{k}{\beta_1^k}.
	\end{equation}
	Now, by using an induction method, we infer that 
	$$F_{k}\leq \beta_1^{k-1}|F_1|+\lambda_{k-1}+\beta_1\lambda_{k-2}+\beta_1^2\lambda_{k-3}+...+\beta_1^{k-2}\lambda_1.$$
	Therefore, by using the last assertion together with \eqref{desprint41} we obtain
	\begin{eqnarray}
		\ln(\|u\|_{\ell_s^*\beta_{k}})=\frac{F_k}{\beta_k}&\leq& \frac{1}{\beta_1}\left(|F_1|+\sum_{j=1}^{k-1}\frac{\lambda_j}{\beta_1^j}\right)\leq \frac{1}{\beta_1}\left(|F_1|+d\sum_{j=1}^{k-1}\frac{1}{\beta_1^j}+\hat{d}\sum_{j=1}^{k-1}\frac{j}{\beta_1^j}\right)\to d_0, \,\, \mbox{as} \,\, k \to \infty. \nonumber
	\end{eqnarray}
	Hence, we deduce that
	$$\|u\|_\infty\leq e^{d_0}.$$
	Therefore $u$ is in $L^{\infty}(\Omega)$. This finishes the proof. 
	\vskip.3cm
	\noindent{\bf Proof of Theorem \ref{A}-ii):} The proof here is quite similar to the proof of Theorem \ref{B}-ii). In fact, by using the fact that $\Phi$ is equivalent to $t\mapsto |t|^m$, we infer that $W_0^{s,\Phi}(\Omega)=W_0^{s,m}(\Omega)\hookrightarrow L^{m_s^*}(\Omega)$. Consequently, the inequality \eqref{desBK4i} can be rewritten in the following form
	\begin{eqnarray}\label{desBK42iia}
		\xi_\Phi^-\left(\|u u_L^{\beta-1}\|_{m_s^*}\right)&\leq & \overline{D}\beta^m\left[1+R\|u\|_{\beta m}^{\beta m}\right].
	\end{eqnarray}
	Recall also that 
	$$\beta_{k+1}m=\beta_km_s^*\,\,{where}\,\,\beta_1=\frac{m_s^*}{m}=\frac{N}{N-sm}.$$
	Putting all together we obtain
	\begin{eqnarray}\label{desBK5iia}
		\xi_\Phi^-(\|u\|_{{\beta_{k+1}}m_s^{*}}^{\beta_{k+1}})\leq \overline{D}\beta_{k+1}^m\left(1+R\|u\|_{\beta_k m^*_s}^{\beta_k m^*_s}\right).
	\end{eqnarray}
	At this stage, by using the same steps explored in the proof of Theorem \ref{A}-i) we obtain that $u \in L^{\infty}(\Omega)$. This ends the proof.
	

\end{document}